\documentclass[12pt,a4paper]{amsart}
\usepackage{amsmath}
\usepackage{amssymb}
\usepackage{amsthm}
\usepackage{latexsym}
\usepackage{pstricks}
\usepackage{amsbsy}

\usepackage{array}
\usepackage[all]{xy}
\usepackage{amscd}
\usepackage{mathrsfs}

\usepackage{txfonts}
\usepackage{amsfonts}
\usepackage{graphicx}
\usepackage{listings}
\usepackage{url}
\usepackage{bookmark}

\newtheorem{theorem}{Theorem}[section]
\newtheorem{lemma}[theorem]{Lemma}
\newtheorem{corollary}[theorem]{Corollary}
\newtheorem{proposition}[theorem]{Proposition}
\newtheorem{definition}[theorem]{Definition}
\newtheorem{remark}[theorem]{Remark}

\newtheorem{conjecture}[theorem]{Conjecture}

\newcommand{\abs}[1]{\left\lvert#1\right\rvert}

\definecolor{darkgreen}{rgb}{0.33, 0.42, 0.18}

\providecommand{\Gen}{\mathop{\rm Gen}\nolimits}%

\def\C{\mathcal{C}}

\def\X{\mathcal{X}}

\def\E{\mathcal{E}}

\def\F{\mathcal{F}}
\def\SS{\mathcal{S}}
\def\T{\mathcal{T}}

\def\V{\mathsf{V}}
\def\W{\mathsf{W}}

\providecommand{\rk}{\mathop{\rm rk}\nolimits}%
\providecommand{\End}{\mathop{\rm End}\nolimits}%
\providecommand{\Ext}{\mathop{\rm Ext}\nolimits}%
\providecommand{\Hom}{\mathop{\rm Hom}\nolimits}%
\providecommand{\ind}{\mathop{\rm ind}\nolimits}%
\providecommand{\Fac}{\mathop{\rm Gen}\nolimits}%
\providecommand{\Filt}{\mathop{\rm Filt}\nolimits}%
\newcommand{\module}{\mathop{\rm mod}\nolimits}%
\providecommand{\perpe}{\perp_{0,1}}%

\usepackage{marginnote}

\begin{document}
	
	\title{Mutating signed $\tau$-exceptional sequences}

	\author[Buan]{Aslak Bakke Buan}
	\address{
		Department of Mathematical Sciences \\
		Norwegian University of Science and Technology \\
		7491 Trondheim \\
		Norway \\
	}
	\email{aslak.buan@ntnu.no}
	
	\author[Marsh]{Bethany Rose Marsh}
	\address{School of Mathematics \\ 
		University of Leeds \\ 
		Leeds, LS2 9JT \\ 
		United Kingdom \\
	}
	\email{B.R.Marsh@leeds.ac.uk}

	\begin{abstract}
		\sloppy  We establish some properties of $\tau$-exceptional sequences for finite-dimensional algebras. In an earlier paper we established a bijection between the set of ordered support $\tau$-tilting modules and the set of complete signed $\tau$-exceptional sequences. We describe the action of the symmetric group on the latter induced by its natural action on the former. Similarly, we describe the effect on a $\tau$-exceptional sequence obtained by mutating the corresponding ordered support $\tau$-tilting module via a construction of Adachi-Iyama-Reiten.
	\end{abstract}

\dedicatory{Dedicated to the memory of Helmut Lenzing}
	
	\thanks{
This work was supported by FRINAT grant number 301375 from the
Norwegian Research Council.  The authors would like to thank the Isaac Newton Institute for Mathematical Sciences, Cambridge, for support and hospitality during the programme Cluster Algebras and Representation Theory where work on this paper was undertaken. The first named author would also like to thank the Centre of Advanced Study, Oslo for support and hospitality during the programme Representation Theory: Combinatorial Aspects and Applications. }
	
	\maketitle
	
\section*{Introduction}

The usual notion of exceptional sequences in a module category \sloppy over a finite-dimensional algebra~\cite{cb-exc,ringel-exc} has some drawbacks. In particular, for some non-hereditary algebras, complete exceptional sequences do not exist (see e.g.~\cite[Introduction]{bm-exc}).
In \cite{bm-exc}, we introduced the notion of $\tau$-exceptional sequences, motivated by $\tau$-tilting theory~\cite{air}. Such sequences can be regarded as an alternative generalisation of exceptional sequences to the non-hereditary case with the property that complete $\tau$-exceptional sequences always exist. We also introduced signed $\tau$-exceptional sequences, motivated by the concept of signed exceptional sequences for hereditary algebras~\cite{igusatodorov}, and the link to picture groups \cite{igusatodorov, itw}.

The aim of this paper is to establish further properties of (signed) $\tau$-exceptional sequences, which we now proceed to discuss in more detail.
Recall that a subcategory of a module category is said to be a \emph{wide subcategory} if it is closed under kernels, cokernels and extensions (and therefore inherits an abelian structure).
Let $\Lambda$ be the path algebra of a quiver with $n$ vertices and $\module \Lambda$ the category of finite-dimensional left $\Lambda$-modules.
A $\Lambda$-module $X$ is said to be \emph{exceptional} if $\Ext^1(X,X)=0$. If $X$ is exceptional, then the subcategory $X^{\perpe}$ consisting of modules $Y$ such that
$\Hom(X,Y)=0$ and $\Ext^1(X,Y)=0$ is known as the \emph{perpendicular category} of $X$~\cite[\S1]{geiglelenzing}.  By~\cite[Prop. 1.1]{geiglelenzing},~\cite[Thm. 2.3]{schofield}, $X^{\perpe}$ is a wide subcategory of $\module \Lambda$ equivalent to the module category of the path algebra of a quiver with $n-1$ vertices.
An \emph{exceptional sequence} is a sequence $(X_1,\ldots ,X_r)$ where $X_r$ is an indecomposable exceptional $\Lambda$-module and $(X_1,\ldots ,X_{r-1})$ is an exceptional sequence in $X_r^{\perpe}$.

If $r=n$, then an exceptional sequence $(X_1,X_2,\ldots ,X_r)$ is said to be \emph{complete}. Note that a complete exceptional sequence gives rise to a flag of wide subcategories
$$\C_n\subseteq \C_{n-1}\subseteq \cdots \subseteq \C_0=\module \Lambda$$
where $\C_i=(X_1\amalg X_2\amalg\cdots \amalg X_i)^{\perpe}$.

The article~\cite{igusatodorov} introduced the notion of a \emph{signed exceptional sequence}. Let $D^b(\Lambda)$ denote the bounded derived category of $\module \Lambda$. Since $\Lambda$ is hereditary,  every indecomposable
object in $D^b(\Lambda)$ is of the form $X[i]$, where $[i]$ denotes the $i$th power of the shift and $X$ is an indecomposable $\Lambda$-module. We write $|X[i]|=X$.

A signed exceptional sequence in $\module \Lambda$ is a sequence $(X_1,X_2,\ldots ,X_r)$ of indecomposable objects in $D^b(\Lambda)$ which are each of the form $Y[j]$ for $j=0$ or $j=1$ for some $\Lambda$-module $Y$, where $X_i=Y[1]$ is allowed only if $|X_i|$ is relatively projective in $(|X_{i+1}|\amalg |X_{i+2}|\amalg \cdots \amalg |X_r|)^{\perpe}$, and where
$(|X_1|,|X_2|,\ldots, |X_r|)$ is an exceptional sequence.

In~\cite{igusatodorov}, signed exceptional sequences were introduced in order to define the \emph{cluster morphism category} of $\Lambda$, whose objects are the wide subcategories of $\module \Lambda$. The morphisms are described by the signed exceptional sequences. It is shown that the classifying space of the cluster morphism category is a $K(\pi,1)$, where $\pi$ is the \emph{picture group}~\cite{itw} of $\Lambda$.

Now let $\Lambda$ be an arbitrary finite-dimensional algebra over a field. Suppose that $\Lambda$ has $n$ simple modules. 
The linchpin of the definition of $\tau$-exceptional sequence is the notion of a $\tau$-perpendicular category~\cite[\S1]{jasso}, which plays the role of the Geigle-Lenzing perpendicular category in the general case. A $\Lambda$-module $X$ is said to be \emph{$\tau$-rigid} if $\Hom(X,\tau X)=0$. Then the $\tau$-perpendicular category of $X$ is the subcategory $J(X)$ consisting of modules $Y$ such that $\Hom(X,Y)=0$ and $\Hom(Y,\tau X)=0$. By~\cite[Cor. 3.22]{bst},~\cite[Thm. 4.12]{DIRRT}
$J(X)$ is a wide subcategory of $\module \Lambda$.
By~\cite[Thm. 3.8]{jasso} $J(X)$ is equivalent to the module category of an algebra, which has $n-1$ simple modules if $X$ is indecomposable.
A $\tau$-exceptional sequence in $\module \Lambda$ is a sequence $(X_1,X_2,\ldots ,X_r)$ of $\Lambda$-modules where $X_r$ is $\tau$-rigid and $(X_1,\ldots ,X_{r-1})$ is a $\tau$-exceptional sequence in $J(X_r)$ (regarded as a module category).
Signed $\tau$-exceptional sequences are then defined in a similar way to signed exceptional sequences (see above, or Section~\ref{sec:background}). A $\tau$-exceptional sequence, or signed $\tau$-exceptional sequence, is said to be {\em complete} if it has $n$ terms. Clearly, complete $\tau$-exceptional sequences and signed $\tau$-exceptional sequences exist for any finite-dimensional algebra.
In \cite{mt}, an interesting interpretation in terms of strandardly stratifying systems was given.

Also, $\tau$-exceptional sequences were used in~\cite{bm-wide} to define the morphisms in the $\tau$-cluster morphism category of the module category of a $\tau$-tilting-finite algebra, whose objects are the wide subcategories of the module category. This was extended to an arbitrary finite-dimensional algebra in~\cite{bh}. In~\cite[Thm. 4.16]{hansonigusa}, it was shown that, if $\Lambda$ is a Nakayama algebra, the classifying space of the cluster morphism category is a $K(\pi,1)$ for the picture group~\cite{itw} of $\Lambda$.

In this paper, we study some properties of $\tau$-exceptional sequences. In Section~\ref{sec:uniqueness}, we prove our first main result, which is restricted to the case of $\tau$-tilting finite algebras, i.e algebras with a finite number of basic $\tau$-tilting modules. Under this assumption, we show that if
$(X_1,X_2,\ldots ,X_i,\ldots X_n)$ and $(X_1,X_2,\ldots ,X'_i,\ldots ,X_n)$ are complete $\tau$-exceptional sequences, then $X'_i \cong X_i$. We conjecture that this result holds without this assumption.

Suppose now that $\Lambda$ is an arbitrary finite-dimensional algebra.
In~\cite[Thm. 5.4]{bm-exc}, it was shown that there is a bijection between complete signed $\tau$-exceptional sequences in $\module \Lambda$ and ordered support $\tau$-tilting objects in $\module \Lambda$. Here a support $\tau$-tilting object is a pair $(P,M)$ where $P$ is projective, $M$ is $\tau$-rigid and $\Hom(P,M)=0$, and an ordered support $\tau$-tilting object is an ordering of the indecomposable summands of $P$ and $M$ (retaining the information as to whether each object is a summand of $P$ or $M$).
Thus the symmetric group acts naturally on the set of ordered support $\tau$-tilting objects and hence, via the bijection, on the set of complete signed $\tau$-exceptional sequences. In Section~\ref{sec:transposition}, we give an explicit description of the action of a simple transposition.

Support $\tau$-tilting objects can be mutated (see~\cite[\S 2.3]{air}), and thus so can ordered $\tau$-tilting objects. In Section~\ref{sec:mutation}, we describe the effect on the corresponding complete $\tau$-exceptional sequences, translated via the bijection above. We also combine the action of the symmetric group and mutations to give an action of the larger \emph{mutation group} considered in~\cite{kingpressland}.

Since the braid group on $n$ strands acts transitively on the set of exceptional sequences over a hereditary algebra~\cite{cb-exc,ringel-exc} with $n$ simple modules up to isomorphism, a natural question is whether this braid group acts transitively on the set of all signed $\tau$-exceptional sequences. In Section~\ref{sec:examples} we show that, for the Kronecker algebra, 
there is no transitive action of the braid group on $2$ strands (i.e. the infinite cyclic group)
on the set of signed $\tau$-exceptional sequences which factors through the action of the mutation group referred to above on the set of such sequences.
We also give an example showing that the obvious generalisation of the definition of the braid action on exceptional sequences to the (signed) $\tau$-exceptional case does not work, at least without substantial modification.

\section{Background}
\label{sec:background}
Let $\Lambda$ be a finite-dimensional basic algebra,
and denote by $\module \Lambda$ the category of finite-dimensional left $\Lambda$-modules.
We let $\tau$ denote the Auslander-Reiten translate on $\module \Lambda$.
We assume any subcategories $\X$
to be full and closed under isomorphism;
we define
$\X^{\perp} = \{Y \in \module \Lambda \mid \Hom(\X, Y) = 0\}$
and define ${^{\perp}\X}$ dually.

Consider $\C(\module \Lambda) = \module \Lambda \oplus \module \Lambda[1]$ as a full subcategory of the bounded derived category $D^b(\module \Lambda)$.
For an indecomposable object $U$ in $\C(\module \Lambda)$, we set
$\abs{U} = U$ if $U$ is in $\module \Lambda$ and $\abs{U} = U[-1]$ if $U$ is in $\module \Lambda[1]$. 
If $U$ in $\C(\module\Lambda)$ is basic, we denote by $\rk(U)$ the number of indecomposable summands of $U$.

We recall some notions from~\cite[\S0]{air} (in some cases stated slightly differently, but equivalently).
A $\Lambda$-module $M$ is called {\em $\tau$-rigid} if $\Hom(M,\tau M) = 0$. A (usually assumed basic) object $M \oplus P[1]$ in $\C(\module \Lambda)$ is said to be a
{\em support $\tau$-rigid object} if $M$ is a $\tau$-rigid $\Lambda$-module and $P$ is a projective $\Lambda$-module with $\Hom_\Lambda(P, M)= 0$.
An object $U = M \oplus P[1]$ is said to be a \emph{support $\tau$-tilting object} if $\rk(U) =\rk(\Lambda)$. If, in addition, $P=0$, $U$ is said to be a \emph{$\tau$-tilting module}.

Recall that a subcategory $\W$ of $\module\Lambda$ is said to be \emph{wide} if it is closed under kernels, cokernels and extensions. If a wide subcategory $\W$ of $\module \Lambda$ is equivalent to a module category
$\module \Lambda'$, we set $\rk \W \coloneqq \rk \Lambda'$.

Objects which are $\tau$-rigid give rise to a particular class of wide subcategories.

\begin{definition}\cite[Defn. 3.3]{jasso}
For a support $\tau$-rigid object $U = M \oplus P[1]$ in $\C(\module \Lambda)$
the category
$$J(U) =  (M \oplus P)^{\perp} \cap {^\perp{(\tau M)}}$$
is called a $\tau$-perpendicular subcategory.
\end{definition}

In the following Theorem, (a) is from~\cite[Thm. 4.12]{DIRRT},~\cite[Cor. 3.22]{bst} and (c) is from \cite[Thm. 3.8]{jasso}.
For (b), see~\cite[Prop. 4.12]{en} and~\cite[Lemma 4.7]{en-sak}.

\begin{theorem}\label{prop:counting}
 A $\tau$-perpendicular subcategory $J(U)$ of $\module \Lambda$ is:
 \begin{itemize}
 \item[(a)] wide; 
 \item[(b)] functorially finite;
 \item[(c)] equivalent to $\module \Lambda_U$ for some finite-dimensional algebra $\Lambda_U$ with
	 $\rk(\Lambda) = \rk(U) + \rk(\Lambda_U)$.
\end{itemize}
\end{theorem}

Let $\W$ be a $\tau$-perpendicular subcategory of $\module \Lambda$. Since
$\W$ is equivalent to a module category, we can also consider the $\tau$-tilting
theory of $\W$. Let $\C(\W) = \W \oplus \W[1]$, as a subcategory of $D^b(\module \Lambda)$. Note that since $\W$ is an exact
subcategory of $\module \Lambda$, there is a canonical isomorphism 
$$\Hom_{D^b(\W)}(X,Y[1]) \simeq \Hom_{D^b(\module \Lambda)}(X,Y[1]),$$
for modules $X,Y$ in $\W$,
so we can also consider $\C(\W)$ as a subcategory of $D^b(\W)$.

Note that in general $\tau_\W X \not \simeq \tau X$ for a module $X$ in $\W$,
and hence there in general there exist modules which are $\tau$-rigid in $\W$
but not $\tau$-rigid in $\module\Lambda$. But we do have the following.
For a support $\tau$-rigid object $V = N \oplus Q[1]$ in $\C(\W) \subset \C(\module \Lambda)$, 
set
$$J_\W(V) =  (N \oplus Q)^{\perp} \cap {^{\perp} {(\tau_{\W} N)}}\cap \W.$$

The following useful Lemma follows from~\cite[Prop. 5.8]{auslandersmalo} (see also~\cite[Prop. 1.2]{air}).

\begin{lemma}\label{lem:extcond}
	Let $\W$ be a $\tau$-perpendicular subcategory of $\module \Lambda$, and assume $X, Y$ lie in $\W$.
	Then the  following hold:
	\begin{itemize}
		\item[(a)] $\Hom(Y, \tau_{\W} X) = 0$ if and only if $\Ext^1(X, \W \cap \Gen Y) = 0$.
		\item[(b)] $X$ is $\tau$-rigid in $\W$ if and only if $\Ext^1(X, \W \cap \Gen X) =0$. 
	\end{itemize}
\end{lemma}

\begin{lemma}\label{lem:inc}
	Let $\W' \subseteq \W$ be $\tau$-perpendicular subcategories of $\module \Lambda$,
	and let $X$ be an object in $\W'$.
	\begin{itemize}
		\item[(a)] If $X$ is $\tau$-rigid in $\W$, then $X$ is also $\tau$-rigid in $\W'$.
		\item[(b)] We have $J_{\W}(X) \subseteq J_{\W'}(X)$.
	\end{itemize}		
\end{lemma}	

\begin{proof}
	\begin{itemize}
		\item[(a)] By Lemma \ref{lem:extcond}, $X$ is $\tau$-rigid in $\W$ implies
		$\Ext^1(X,\W \cap \Gen X ) = 0$. 
		Hence also $\Ext^1(X,\W' \cap \Gen X ) = 0$, and 
		applying Lemma \ref{lem:extcond}  again we obtain that $X$ is $\tau$-rigid in $\W'$. 
		\item[(b)] If $Y$ is in $J_{\W}(X)$, then we have that 
		$\Ext^1(X,\W \cap \Gen Y ) = 0$ by Lemma \ref{lem:extcond}. Hence also 
		$\Ext^1(X,\W' \cap \Gen Y ) = \Ext^1_{\W'}(X, \Gen Y)= 0$. We have $\Hom(X,Y)= 0$ by
		assumption. Hence $Y$ lies in  $J_{\W'}(X)$.
	\end{itemize}	
\end{proof}

The following bijection is crucial. It was proved in \cite{bm-exc}, and can be seen as a refinement
of~\cite[Thm. 3.16]{jasso}.

\begin{theorem}\cite[Prop. 5.6]{bm-exc} \label{bijection}
	Let $\W$ be a $\tau$-perpendicular subcategory of $\module\Lambda$, 
	and let $U$ be a support $\tau$-rigid object in $\C(\W)$.
	Then there is a bijection $\E^{\W}_{U}$ from
	$$\{X \in \ind(\C(\W)) \mid X \oplus U \text{  support } \tau\text{-rigid in } \C(\W)\}
	\setminus \ind U $$
	to
	$$\{X \in  \ind(\C(J_{\W}(U)) \mid \text{$X$ is support $\tau$-rigid in $\C(J_{\W}(U))$} \}.$$
\end{theorem}

We denote the inverse of $\E^{\W}_{U}$ by $\F^{\W}_{U}$ and, when $\W = \module \Lambda$, we denote the map in Theorem \ref{bijection} and its inverse simply by $\E_{U}$ and $\F_{U}$.

Using the bijection in Theorem \ref{bijection}, the following was proved in 
 \cite[Thms. 1.4, 1.7]{bm-wide} for the $\tau$-tilting finite case.
It was generalised in \cite[Thms. 6.4, 6.12]{bh} to arbitrary finite-dimensional algebras.

\begin{theorem}\label{thm:formulas} \cite{bh,bm-wide}
Let $U \oplus V$ be a $\tau$-rigid object in $\C(\module \Lambda)$
\begin{itemize}
\item[(a)] We have $J_{J(U)}\, (\E_{U}(V))= J(U \oplus V)$.
\item[(b)] We have $\E_{U \oplus V} = (\E^{J(U)}_{\E_{U}(V)})\, \E_{U}$.
\end{itemize}
\end{theorem}

We also recall the following.

\begin{lemma} \label{lem:widewide} \cite[Lemma 4.5]{bm-wide}
 Let $\V$ and $\W$ be wide subcategories of $\module \Lambda$ with $\V \subseteq \W$. Then $\V$ is a wide subcategory of $\W$.
\end{lemma}

\section{Uniqueness}\label{sec:uniqueness}
Let $n$ be the number of simple $\Lambda$-modules.
Recall that a complete $\tau$-exceptional sequence in $\module \Lambda$ is a sequence $(X_1,X_2,\ldots ,X_n)$ of indecomposable $\Lambda$-modules where $X_n$ is $\tau$-rigid and $(X_1,\ldots ,X_{n-1})$ is a $\tau$-exceptional sequence in $J(X_n)$.
Moreover, a sequence $(X_1,X_2,\ldots ,X_n)$ of indecomposable objects 
in $\C(\module \Lambda)$ is a signed $\tau$-exceptional sequence, if   
(i) $X_n$ is either a $\tau$-rigid module or of the form $P[1]$ for some projective $\Lambda$-module $P$, and (ii) $(X_1,X_2,\ldots ,X_{n-1})$
is a signed $\tau$-exceptional sequence in $J(\abs{X_n})$. Note that this means that $X_{n-1}$ is either $\tau$-rigid in $J(\abs{X_n})$ (i.e. $\tau$-rigid in the equivalent module category), or $X_{n-1}= P'[1]$, where $P'$ is
(relative) projective in $J(\abs{X_n})$, and so on.

Recall that $\Lambda$ is said to be \emph{$\tau$-tilting finite} if it only has a finite number of indecomposable $\tau$-rigid modules.
In this section, we shall prove the following 
uniqueness result for $\tau$-exceptional sequences over such algebras:

\begin{theorem}\label{thm:unique}
	Let $\Lambda$ be a $\tau$-tilting finite algebra. Then the following
	hold.
	\begin{itemize}
	    \item[(a)] Let $(A_1, A_2, \dots, A_n)$ and $(B_1, B_2, \dots, B_n)$ be complete $\tau$-exceptional 
	sequences in $\module \Lambda$. If, for some $t \in \{1, \dots, n\}$, we have $A_i= B_i$ for
	all $i \neq t$, then also $A_t = B_t$.
	    \item[(b)] Let $(A_1, A_2, \dots, A_n)$ and $(B_1, B_2, \dots, B_n)$ be complete signed $\tau$-exceptional 
	sequences in $\C(\module \Lambda)$. If, for some $t \in \{1, \dots, n\}$, we have $\abs{A_i}= \abs{B_i}$ for
	all $i \neq t$, then also $\abs{A_t} = \abs{B_t}$.
	\end{itemize}
\end{theorem}

We first recall the following:
 
\begin{theorem} \label{allwidejasso}
	Let $\Lambda$ be a $\tau$-tilting finite algebra. Then the following hold
	for any wide subcategory $\W$ of $\module \Lambda$.
\begin{itemize}
	\item[(a)]\cite[Theorem 4.18]{DIRRT}
 We have $\W = J(U)$ for some support $\tau$-rigid object $U$ in $\C(\module \Lambda)$.
\item[(b)] \cite[Thm. 3.8, Thm. 3.16]{jasso} The wide subcategory $\W$ is $\tau$-tilting finite.
\end{itemize}
\end{theorem}

We next make the following observation, which holds for all finite-dimensional algebras:

\begin{lemma}\label{lem:abs}
If $(A_1, A_2, \dots, A_n)$ is a complete
 signed $\tau$-exceptional 
	sequence, then $(\abs{A_1}, \abs{A_2}, \dots, \abs{A_n})$ is a complete (unsigned) $\tau$-exceptional sequence.
\end{lemma}
 
\begin{proof}
We first claim that $(A_1, A_2, \dots, A_{n-1}, \abs{A_n})$
is a signed $\tau$-exceptional sequence.
If $\abs{A_n}$ is projective, then $J(\abs{A_n}[1]) = J(\abs{A_n})$. Hence the initial claim follows from the definition of signed $\tau$-exceptional sequences.
The same argument gives that also 
$(A_1, A_2, \dots, A_{n-2}, \abs{A_{n-1}}, \abs{A_n})$
is a $\tau$-exceptional sequence, and so on.
\end{proof}

It is clear that Lemma~\ref{lem:abs} and Theorem~\ref{thm:unique} (a), imply Theorem~\ref{thm:unique} (b), so it is enough to prove Theorem~\ref{thm:unique} (a).

In the remainder of this section, we will prove Theorem \ref{thm:unique} (a). So, we assume for the remainder of the section that $\Lambda$ is $\tau$-tilting finite.

We then have the following:

\begin{lemma} \label{lem:subcategoryrank}
Let $\W$, $\W'$ be wide subcategories of $\module \Lambda$ with $\W' \subseteq \W$. Then: 
\begin{itemize}
    \item [(a)] We have $\rk \W'\leq \rk \W$;
    \item [(b)] If $\rk \W=\rk \W'$ then $\W=\W'$.
\end{itemize}
\end{lemma}
\begin{proof}
For (a), we note that, by Lemma~\ref{lem:widewide}, $\W'$ is a wide subcategory of $\W$. By Theorem~\ref{allwidejasso}, $\W'$ is of the form $J_{\W}(U)$ for some $\tau$-rigid object $U$ in $\W$. Hence, by Theorem~\ref{prop:counting},
$\rk \W'=\rk \W-r\leq \rk \W$, where $r$ is the number of non-isomorphic indecomposable direct summands of $U$.

For (b) suppose, in addition, that $\rk\W=\rk\W'$. Then $r=0$ in the above, so $U=0$ and we have $\W=\W'$ as required.
\end{proof}

We give an alternative proof of (a) at the end of this section.

\begin{lemma}\label{lem:uniqueW}
Let $\W$ and $\W'$ be wide subcategories of $\module \Lambda$, and assume \sloppy  that
$(A_1, A_2, \dots, A_m)$ is a complete  $\tau$-exceptional sequence in both 
$\W$ and $\W'$. Then we  have $\W = \W'$. 
\end{lemma}

\begin{proof}
By assumption, $\rk \W = m = \rk \W'$ so, by Lemma~\ref{lem:subcategoryrank} (a), we have $\rk (\W \cap \W') \leq m$.	
By Lemma \ref{lem:inc} (a), we have that $A_m$ is $\tau$-rigid in $\W\cap \W'$.
We then have $\rk J_{\W} (A_m) = m-1$ and $\rk J_{\W \cap \W'} (A_m) \leq m-1$.
By Lemma \ref{lem:inc} (b), we have that $J_{\W}(A_m) \subseteq J_{\W \cap \W'}(A_m)$.
Hence we must have $\rk J_{W}(A_m) = \rk  J_{\W \cap \W'}(A_m)$, and it follows that 
$\rk \W = \rk J_{W}(A_m)  +1 = \rk  J_{\W \cap \W'}(A_m) +1 = \rk \W \cap \W'$.
By Lemma~\ref{lem:subcategoryrank} (b), we have
$\W = \W \cap \W' = \W'$.
\end{proof}	

\begin{lemma}\label{lem:uniqueA}
If $(A_1, A_2, \dots, A_{n-1}, A_n)$ and 
	$(A_1, A_2, \dots, A_{n-1}, A'_n)$ are \sloppy complete signed $\tau$-exceptional sequences
	in $\module \Lambda$, then $A_n = A'_n$.
\end{lemma}	

\begin{proof}
Let $\W= J(A_n)$ and $\W' = J(A'_n)$. Then the statement follows  	
as a direct consequence of \cite[Proposition 10.7]{bm-wide} and Lemma \ref{lem:uniqueW}.
\end{proof}		

We can now complete the proof of Theorem \ref{thm:unique} (a), and hence the main theorem of this section.

\begin{proof}[Proof of Theorem \ref{thm:unique} (a)]
If $t= n$, this follows directly from Lemma \ref{lem:uniqueA}.
Assume $t \in \{1, \dots, n-1\}$. Let $W^A_n = J(A_n)$, and for $j \in \{t, \dots, n-1\}$ define
recursively $\W^A_j = J_{\W^A_{j+1}}(A_j)$.
Define similarly $\W^B_n = J(B_n)$ and $\W^B_j = J_{\W^B_{j+1}}(B_j)$.
Then $\W^A_{t+1} = \W^B_{t+1} \coloneqq  \W'$, and  
$\W' \simeq \module \Lambda'$ for a finite-dimensional algebra $\Lambda'$,
and we have that 
$(A_1, A_2, \dots, A_{t-1}, A_t)$ and $(B_1, B_2, \dots, B_{t-1}, B_t)
= (A_1, A_2, \dots, A_{t-1}, B_t)$ are complete exceptional sequences in $\W'$.
Hence, we obtain that 
$A_t = B_t$ by Lemma \ref{lem:uniqueA}.	
\end{proof}	

We note that the uniqueness property of Theorem \ref{thm:unique} also holds for arbitrary finite-dimensional hereditary algebras, 
by \cite[Lemma 2]{cb-exc}, and we conjecture that the assumption on $\tau$-tilting finiteness should not be 
necessary.

\begin{conjecture}\label{conj}
The statement of Theorem \ref{thm:unique} holds for all
finite-dimensional algebras.
\end{conjecture}

An alternative proof of Lemma~\ref{lem:subcategoryrank} (a) can be given using the theory of bricks.
Recall that a $\Lambda$-module $M$ is called a \emph{brick} if $\End(M)$ is a division algebra. A set of isoclasses of pairwise Hom-orthogonal bricks is called a \emph{semibrick}.
Let $C$ be a full subcategory of $\module \Lambda$. We denote by $T(C)$ the smallest torsion class containing $C$, by $\Gen(C)$ the collection of modules obtained as quotients of finite direct sums of modules in $C$, and by $\Filt(C)$ the category of modules with filtrations by modules in $C$.
By the argument in~\cite[Lemma 3.1]{ms}, $T(C)=\Filt(\Fac(C))$.
A semibrick, $S$, is called \emph{left finite}~\cite[Definition 
1.2]{asai} if $T(S)$ is functorially finite.
Let $n_{\Lambda}$ denote the number of isomorphism classes of simple $\Lambda$ modules.
We recall:

\begin{proposition} \label{prop:bricksize} \cite[1.10]{asai}
If $S$ is a left finite semibrick, then $|S|\leq n_{\Lambda}$.
\end{proposition}

\begin{proposition} \label{prop:torsionfinite} \cite[1.2]{DIJ}
Let $A$ be a $\tau$-tilting finite algebra, and let $T$ be a torsion class in $\module A$. Then $T$ is functorially finite.
\end{proposition}

\begin{proof}[Proof of Lemma~\ref{lem:subcategoryrank} (a)]
Let $S$ be the set of isoclasses of simple objects in $\W'$. Then $S$ is a semibrick in $\W$. Note that, by~\cite[Thm. 3.8]{jasso}, $\W\simeq \module \Lambda'$ for some finite-dimensional algebra $\Lambda'$.
By \cite[Prop. 4.2(b)]{bm-wide}, $\Lambda'$ is $\tau$-tilting finite.
Hence, by Proposition~\ref{prop:torsionfinite}, $S$ is left-finite in $\W$.
By Proposition~\ref{prop:bricksize}, $\rk(\W')=|S|\leq \rk(\W)$.
\end{proof}

\section{Transposition}
\label{sec:transposition}
We now return to the general case, where $\Lambda$ is an arbitrary finite-dimensional algebra. 
A sequence $(T_1, T_2, \dots, T_r)$ of indecomposable objects in $\C(\module \Lambda)$ is called an {\em ordered $\tau$-rigid object} if 
$\oplus_{i=1}^r T_i$ is a $\tau$-rigid object, 
and an {\em ordered support $\tau$-tilting object}
if $r=n \colon = \rk \Lambda$.
The symmetric group acts on the set of ordered support $\tau$-tilting objects 
in a wide subcategory $\W$ of $\module \Lambda$, by reordering. 
We recall the following theorem from \cite{bm-exc}.

\begin{theorem} \cite[Thm. 5.4]{bm-exc} \label{bijection2}
	For each $\tau$-perpendicular subcategory $\W$ of $\module\Lambda$, there are mutually inverse bijections
	$$\{\text{ordered support $\tau$-tilting objects in $\W$ } \}$$
	$$\psi^{\W}  \downarrow \text{  } \uparrow \phi^{\W}$$
	$$\{\text{complete signed $\tau$-exceptional sequences in $\W$} \}$$
\end{theorem}

In the case $\W=\module\Lambda$, we write $\psi$ for  $\psi^{\module \Lambda}$  and $\phi$ for $\phi^{\module \Lambda}$. In this section, we will describe the action of the symmetric group on complete signed $\tau$-exceptional sequences in a 
$\tau$-perpendicular subcategory $\W$ of $\module\Lambda$ induced by the bijections above.

\begin{remark}\label{rem:formulas}
If $\psi(T_1, T_2, \dots, T_n) = (A_1, \dots, A_n)$, let
\begin{align*}
 \W_{n-1} & = J(A_n),\\
 \W_{n-2} &= J_{\W_{n-1}}(A_{n-1}), \\
 & \, \vdots  \\
 \W_j & = J_{\W_{j+1}}(A_{j+1}), \\
 & \, \vdots  \\
 \W_1 &= J_{\W_{2}}(A_2).
 \end{align*}
 
 Then we have
	\begin{align*}
	A_n &= T_n, \\
	A_{n-1} &= \E_{A_n}(T_{n-1}), \\
	A_{n-2} &= \E^{\W_{n-1}}_{A_{n-1}}  \E_{A_n} (T_{n-2}), \\
	& \, \vdots  \\
	A_{n-j} &= \E_{A_{n-j+1}}^{\W_{n-j+1}} \dots \E^{\W_{n-1}}_{A_{n-1}}   \E_{A_n}(T_{n-j}), \\
	&\, \vdots  \\
	A_1 &= \E_{A_2}^{\W_2} \dots \E^{\W_{n-1}}_{A_{n-1}}    \E_{A_n}(T_1).
\end{align*}
\end{remark}

\begin{lemma}\label{lem:sum}
With notation as above we have 
$$\W_{n-i} = J(T_n \oplus T_{n-1} \oplus \dots \oplus T_{n-{i+1}}),$$
for $i= 1, \dots, n-1$.	
\end{lemma}

\begin{proof}
This follows from repeated use of Theorem \ref{thm:formulas} (a).
\end{proof}

Let $\Lambda$ be an algebra of rank $n$, and
let $\T_o^{\Lambda}$ be the set of ordered basic support $\tau$-tilting objects in $\C(\module \Lambda)$.
There is a natural action of the symmetric group
$S_n$ on $T_o^{\Lambda}$,
given by $\pi_i (T_1, \dots, T_n) = (T_1, \dots, T_{i-1}, T_{i+1}, T_i, T_{i+2}, \dots, T_n)$, where $\pi_i$ denotes the simple transposition $(i\ i+1)$.

\begin{theorem}\label{thm:symm-action}
\begin{itemize}
\item[(a)] If $\SS = (A_1, \dots, A_n)$ is a signed $\tau$-exceptional sequence, then
for $i \in  \{1, \dots, n-1\}$ we have that
$$\widetilde{\pi_i}(\SS) \coloneqq (A_1, \dots, A_{i-1}, \E^{(i)}
(A_{i+1}), \F^{(i)}(A_i), A_{i+2}, \dots, A_n)$$	
	is a signed $\tau$-exceptional sequence, where
	$$\F^{(i)} \coloneqq  \F_{A_{i+1}}^{\W_{i+1}}$$
	and
	$$\E^{(i)} \coloneqq  \E^{\W_{i+1}}_{\F^{(i)}(A_i)}$$
		for $i \in \{1, \dots, n-2\}$ and where
	$\F^{(n-1)} \coloneqq  \F_{A_n}$ and
	$\E^{(n-1)} \coloneqq  \E_{\F^{(n-1)}(A_{n-1})}$. 
	\item[(b)]  For each $i  \in \{1, \dots, n-1\}$ we have $ \widetilde{\pi_i} \psi= \psi \pi_i$.
	\end{itemize}
\end{theorem}

\begin{lemma}\label{lem:commute}
Let $(B,C)$ be an ordered $\tau$-rigid object in $\C(\module \Lambda)$.
Then we have that
$\E^{J(C)}_{\E_C(B)}\E_C = \E^{J(B)}_{\E_B(C)}\E_B$.
\end{lemma}

\begin{proof}
By applying Theorem \ref{thm:formulas} (b) twice, we obtain 
$$ \E^{J(C)}_{\E_C(B)}\E_C = \E_{C \oplus B} 
=  \E_{B \oplus C} 
=   \E^{J(B)}_{\E_B(C)}\E_B.	$$
\end{proof}

\begin{proposition}\label{prop:final}
Let $(T_1, \dots, T_n)$ be an ordered support $\tau$-tilting object in $\C(\module \Lambda)$,
and assume that $\psi(T_1, \dots, T_n) = (A_1, \dots, A_n)$.
Then $$\psi(\pi_{n-1}(T_1, \dots, T_n)) = (A_1, \dots A_{n-2}, \E_{\F_{A_{n}}
	(A_{n-1})}(A_n), \F_{A_n}(A_{n-1})).$$
\end{proposition}	

\begin{proof}
Let  $\ (B_1, B_2, \dots, B_n) = \psi(\pi_{n-1}(T_1, \dots, T_n)) = \psi(T_1, \dots, T_{n-2}, T_n, T_{n-1})$. We need to show that 
\begin{itemize}
	\item[(i)] $B_n = \F_{A_n}(A_{n-1})$,
	\item[(ii)] $B_{n-1} = \E_{\F_{A_{n}}(A_{n-1})}(A_n)$, and
	\item[(iii)] $B_j = A_j$, for $1 \leq j \leq n-2$.
\end{itemize}
	
Note that, by Remark~\ref{rem:formulas}, we have that $B_n=T_{n-1}$, $A_n=T_n$ and
that $A_{n-1} = \E_{T_n}(T_{n-1})$. Hence, we have $\F_{T_n}(A_{n-1}) = \F_{T_n}\E_{T_n}(T_{n-1}) = T_{n-1} =
B_n$, which proves claim (i).
Moreover, it also follows from Remark~\ref{rem:formulas} that $B_{n-1} = \E_{T_{n-1}}(T_n) = \E_{\F_{T_n}(A_{n-1})}(A_n)$,
which proves claim (ii).

It remains to prove that $B_{j} = A_j$ for $j \leq n-2$.
Apply Lemma \ref{lem:commute}, with $B= T_{n-1}$ and $C = T_n$ to obtain that 
$$\E^{J(T_n)}_{\E_{T_n}(T_{n-1})}\E_{T_n} = \E^{J(T_{n-1})}_{\E_{T_{n-1}}(T_n)}\E_{T_{n-1}}.$$

It now follows directly, from Remark \ref{rem:formulas} and Lemma \ref{lem:sum}, that $B_{j} = A_j$ for $j \leq n-2$.
\end{proof}

\begin{proof}[Proof of Theorem \ref{thm:symm-action}]
By Proposition \ref{prop:final}, it follows that both (a) and (b) hold for $i = n-1$.
Assume $i < n-1$. Then $(A_1, \dots, A_{i-1})$ is a complete signed $\tau$-exceptional
sequence for the $\tau$-perpendicular subcategory $\W_{i-1}$, as defined in Remark \ref{rem:formulas}.
Finally, Proposition \ref{prop:final} implies that both (a) and (b) hold also in this case.
\end{proof}

\section{Mutation}
\label{sec:mutation}
For a fixed positive integer $n$, 
consider the group $G_n = \langle \mu_1, \dots , \mu_n \mid \mu_i^2 = e \rangle$ (as in~\cite[\S1]{kingpressland}).
Let $\Lambda$ be a fixed algebra of rank $n$.

Mutation of support $\tau$-tilting objects (as in~\cite[Thm. 2.18]{air}) induces a mutation on the set $\T^{\Lambda}_o$ of ordered basic support $\tau$-tilting objects in $\C(\module \Lambda)$, which can be regarded as an action of $G_n$ on $\T^{\Lambda}_o$ and hence, via the bijections in Theorem~\ref{bijection2}, an action on the set of complete signed $\tau$-exceptional sequences.

The following result follows from~\cite[Thm. 2.18]{air}.

\begin{proposition} \cite[Thm. 2.18]{air}
Let $T = (T_1, \dots, T_n)$ be an ordered support $\tau$-tilting object in $\C(\module \Lambda)$.
Let $i \in \{1, \dots, n\}$. Then there is a unique indecomposable object $T_i^{\ast}$ in 
 $\C(\module \Lambda)$ such that $T(i)=(T_1, \dots, T_{i-1}, T_i^{\ast},
 T_{i+1}, \dots,  T_n)$ is an ordered support $\tau$-tilting object with $T_i^{\ast} \not \simeq T_i$.
\end{proposition}

With $T$ and $T(i)$ as above, we set $\mu_i(T) = T(i)$.
This defines a $G$-action on  $\T^{\Lambda}_o$.
We now describe the corresponding action on the set of complete signed $\tau$-exceptional sequences.

For a complete signed $\tau$-exceptional sequence $\SS = (A_1, \dots, A_n)$, let
$$s_1(\SS) = (\widetilde{A_1}, A_2, \dots, A_n),$$ where 
$$\widetilde{A_1} = 
\begin{cases}
A_1[1] & \text{ if } A_1 \in \module \Lambda; \\
A_1[-1] & \text{ if } A_1 \in \module \Lambda[1].
\end{cases}
$$

Moreover, for $j>1$, let $s_j(\SS) = \widetilde{\pi}_{j-1}  \widetilde{\pi}_{j-2} \dots \widetilde{\pi_1} s_1 \widetilde{\pi}_1
\dots \widetilde{\pi}_{j-2} \widetilde{\pi}_{j-1}(\SS)$.

We make the following observation:

\begin{lemma} \label{lem:rank1}
Let $\Lambda$ be an algebra with a unique simple module. Then the projective cover of the simple module is the unique indecomposable $\tau$-rigid $\Lambda$-module.
\end{lemma}
\begin{proof}
Let $S$ be the unique simple $\Lambda$-module, and suppose that $X$ is a non-projective indecomposable $\tau$-rigid $\Lambda$-module.
Then $\Ext^1(X,M)\not=0$ for some $\Lambda$-module $M$. Since $M$ must be constructed from $S$ by repeated extensions with $S$, it follows that $\Ext^1(X,S)\not=0$. Since $X$ is also constructed from $S$ by repeated extensions with $S$, we have that $S$ is a factor of $X$, so $\Ext^1(X,\Gen X)\not=0$. By Lemma~\ref{lem:extcond}, $X$ is not $\tau$-rigid.
\end{proof}

We now have:

\begin{lemma}\label{lem:first}
	If $\SS = (A_1, A_2, \dots, A_n)$ is a signed $\tau$-exceptional sequence, then so is
	$s_j(\SS)$ for each $j= 1, \dots, n$. Moreover $\SS$ and $s_1(\SS)$ are the only signed $\tau$-exceptional sequences of the form $(X, A_2, \dots, A_n)$ for some object $X$ in $\C(\module\Lambda)$.
	\end{lemma}

\begin{proof}
Define $W^A_j$, for $j=1,\ldots ,n$, as in the proof of Theorem~\ref{thm:unique} (a). By repeated application of Theorem~\ref{prop:counting}, $W^A_j$ is equivalent to a module category over a finite-dimensional algebra of rank $j-1$ for $j=1,\ldots ,n$. Hence $\W^A_2$ is equivalent to the module category of a finite-dimensional algebra with a unique simple module. By Lemma~\ref{lem:rank1}, $\W^A_2$, regarded as a module category, has a unique indecomposable $\tau$-rigid module, given by the unique indecomposable projective module. 
This proves that $s_1(\SS)$ is a signed $\tau$-exceptional sequence, and also that $\SS$ and $s_1(\SS)$ are the only 
signed $\tau$-exceptional sequences of the form $(X, A_2, \dots, A_n)$, i.e. the claim for $j=1$.

The claim for $j>1$ follows by combining this with Theorem \ref{thm:symm-action} (a).
\end{proof}

\begin{proposition}\label{prop:ind-action}
With notation as above 
$s_i \psi = \psi \mu_i$
holds for all $i =1, \dots, n$.
\end{proposition}

\begin{proof}
Consider first the case $i=1$, and assume that $\psi(M_1, M_2, \dots, M_n) = (A_1, A_2, \dots,A_n)$.
We have $\mu_1(M_1, M_2, \dots, M_n) =  (M_1^\ast, M_2, \dots, M_n)$, \sloppy where 
$M^\ast_1 \not \simeq M_1$, and so  $\psi \mu_1(M_1, M_2, \dots, M_n) = \psi(M_1^\ast, M_2, \dots, M_n) =(X, A_2, \dots,A_n)$, for some object $X$. The claim for $i=1$ now follows from Lemma \ref{lem:first}.

For $j>1$, we first note that
$\mu_j =  \pi_{j-1} \pi_{j-2} \dots \pi_1 \mu_1 \pi_1\dots\pi_{j-2}\pi_{j-1}$.

Next we note that, by repeated applications of  
Theorem \ref{thm:symm-action}, we have that

$$\widetilde{\pi_1} \dots \widetilde{\pi_j} \psi  = \psi \pi_1 \dots \pi_j.$$
Combining this with the above we obtain

\begin{align*}
s_j \psi &=  (\widetilde{\pi}_{j-1}  \widetilde{\pi}_{j-2} \dots \widetilde{\pi}_1 s_1 \widetilde{\pi}_1
\dots \widetilde{\pi}_{j-2} \widetilde{\pi}_{j-1}) \psi  \\ 
&=  (\widetilde{\pi}_{j-1}  \widetilde{\pi}_{j-2} \dots \widetilde{\pi}_1 s_1) (\widetilde{\pi}_1
\dots \widetilde{\pi}_{j-2} \widetilde{\pi}_{j-1} \psi)  \\ 
&=  (\widetilde{\pi}_{j-1}  \widetilde{\pi}_{j-2} \dots \widetilde{\pi}_1 s_1) (\psi \pi_1 \dots \pi_{j-2} \pi_{j-1}) \\ 
&=  (\widetilde{\pi}_{j-1}  \widetilde{\pi}_{j-2} \dots \widetilde{\pi}_1 )(s_1 \psi) (\pi_1 \dots \pi_{j-2} \pi_{j-1})   \\
&=  (\widetilde{\pi}_{j-1}  \widetilde{\pi}_{j-2} \dots \widetilde{\pi}_1 )(\psi \mu_1) (\pi_1 \dots \pi_{j-2} \pi_{j-1})   \\
&=  (\widetilde{\pi}_{j-1}  \widetilde{\pi}_{j-2} \dots \widetilde{\pi}_1 \psi) (\mu_1\pi_1 \dots \pi_{j-2} \pi_{j-1})   \\
& = (\psi \pi_{j-1} \pi_{j-2} \dots \pi_1) (\mu_1 \pi_1\dots\pi_{j-2}\pi_{j-1}) \\
& = \psi (\pi_{j-1} \pi_{j-2} \dots \pi_1\mu_1 \pi_1\dots\pi_{j-2}\pi_{j-1}) \\
&= \psi \mu_j .
\end{align*}
\end{proof}	

King and Pressland~\cite[Defn. 1.2]{kingpressland} consider the following group:

\begin{definition} \label{def:mutationgroup} \cite{kingpressland}
Let $M_n = S_n \ltimes G_n$ be the \emph{mutation group} of degree $n$, where $S_n$ acts on $G_n$ via $\sigma(\mu_i)=\mu_{\sigma(i)}$.
\end{definition}

They show that this group acts naturally on labelled (i.e.\ ordered) seeds in a cluster algebra~\cite{fominzelevinskyCA1} via permutation and mutation.
The mutation of support-$\tau$-tilting objects in~\cite[\S2.3]{air} can be regarded as a generalisation of cluster mutation, so it is natural to consider the action of the mutation group in this context. Note that, for an ordered support $\tau$-tilting object $T$, we have 
$\sigma (\mu_i T) = \mu_{\sigma(i)} (\sigma T)$ for a 
permutation $\sigma$ and mutation $\mu_i$, so $M_n$ acts on the set of all ordered support $\tau$-tilting objects, $\T_o^{\Lambda}$. 

We get an induced action of the mutation group on the set of signed $\tau$-exceptional sequences.

\begin{theorem}
	Let $\SS = (A_1, A_2, \dots , A_n)$ be a signed $\tau$-exceptional sequence.
	The operations 
	$$\widetilde{\pi_i}(\SS) \coloneqq (A_1, \dots, A_{i-1}, \E^{(i)}
	(A_{i+1}), \F^{(i)}(A_i), A_{i+2}, \dots, A_n),$$
	$$s_1(\SS) = (\widetilde{A_1}, A_2, \dots , A_n),$$
	and, for $j=2\ldots ,n$,
	$$s_j(\SS) = \widetilde{\pi_{j-1}}  \widetilde{\pi_{j-2}} \dots \widetilde{\pi_1} s_1 \widetilde{\pi_1}
	\dots \widetilde{\pi_{j-2}} \widetilde{\pi_{j-1}} (\SS),$$
	define an action of the mutation group $M_n$ on the set of signed $\tau$-exceptional sequences.
\end{theorem}

\begin{proof}
As already noted, $M_n$ acts on the set of ordered support $\tau$-tilting objects,
and the result hence follows directly from combining Theorem~\ref{thm:symm-action} and Proposition~\ref{prop:ind-action}
with the fact that $\psi$ is a bijection between the set of ordered support $\tau$-tilting objects and the set of signed $\tau$-exceptional sequences (Theorem~\ref{bijection2}).  
\end{proof}	

\section{Examples relating to braid actions}
\label{sec:examples}
Note that the braid group, $B_n$, on $n$ strands, has the symmetric group $S_n$ as a quotient. Since $S_n$ is a subgroup of the mutation group $M_n=S_n \ltimes G_n$, it follows that $B_n$ acts naturally on the set of all ordered support $\tau$-tilting objects $\T_o^{\Lambda}$ and thus on the set of all complete signed $\tau$-exceptional sequences for $\Lambda$, by Theorem~\ref{bijection2}. However, this action is highly non-transitive in general, since the braid group is only permuting the possible orderings of each support $\tau$-tilting object.

It is therefore natural to ask whether there is a transitive action. In the first part of this section, we give an example to show that, at least via the mutation group, this is not possible: we give an algebra for which there is no transitive action of $B_2$ which factors through the action of $M_2$ on $\T_o^{\Lambda}$.

Let $Q$ be the Kronecker quiver
\raisebox{1pt}{$\xymatrix{1 \ar@<-0.5ex>[r] \ar@<0.5ex>[r] & 2}$}, and let $\Lambda$ be the corresponding path algebra. Let $P_i$ (respectively, $I_i$), for $i=1,2$ be the indecomposable projective (respectively, injective) $\Lambda$-modules corresponding to the vertices of $Q$.
Then the $\tau$-tilting (equivalently, tilting) $\Lambda$-modules are the modules
$\tau^{-r}P_1\amalg \tau^{-r}P_2$,
$\tau^{-r}P_1\amalg \tau^{-(r+1)}P_2$,
$\tau^r I_1 \amalg \tau^r I_2$ and
$\tau^{r+1} I_1 \amalg \tau^rI_2$,
for $r=0,1,2,\ldots $.
The support $\tau$-tilting (equivalently, support tilting) objects over $\Lambda$ which are not $\tau$-tilting are $I_1\amalg P_2[1]$, $P_1[1]\oplus P_2$ and $P_1[1]\oplus P_2[1]$. In particular, note that $\Lambda$ is not $\tau$-tilting finite.

The mutation group (see Definition~\ref{def:mutationgroup}) is $M_2=S_2\ltimes G_2$, where $S_2=\{1,\sigma\}$ is the symmetric group of degree $2$ and
$G_2=\langle \mu_1,\mu_2 \ :\ \mu_1^2=\mu_2^2=e\rangle$.
We have $\sigma \mu_1=\mu_2\sigma$ and
$\sigma \mu_2=\mu_1\sigma$. The action of $M_2$ on the set $\T_o^{\Lambda}$ of ordered basic support $\tau$-tilting objects is shown in Figure~\ref{fig:M2action}.

\begin{proposition}
There is no transitive action of the braid group on two strands on $\T_o^{\Lambda}$ which factors through the action of $M_2$ on
$\T_o^{\Lambda}$.
\end{proposition}
\begin{proof}
Note that the braid group $B_2$ on two strands is isomorphic to the infinite cyclic group. If there was a transitive action of $B_2$ on $\T_o^{\Lambda}$ factoring through the action of $M_2$ on $\T_o^{\Lambda}$, then there would be a subgroup of $M_2$ which is a quotient of $B_2$ which acted transitively on $\T_o^{\Lambda}$. Such a subgroup would have to be cyclic.

The elements of $M_2$ are of the form
$$\sigma^\varepsilon \mu_i\mu_{\sigma(i)}\mu_i\cdots \mu_i$$
and
$$\sigma^\varepsilon \mu_i\mu_{\sigma(i)}\mu_i\cdots \mu_{\sigma(i)},$$
where $i\in\{1,2\}$ and $\varepsilon\in\{0,1\}$.
It is easy to check from the description of the action of $M_2$ (see Figure~\ref{fig:M2action}) that, for each of these elements, the (infinite) cyclic group it generates does not act transitively on $\T_o^{\Lambda}$.
\end{proof}
\begin{corollary}
There is no transitive action of the braid group on two strands on the set of all signed $\tau$-exceptional sequences which factors through the action of $M_2$ on
the set of all such sequences.
\end{corollary}

\begin{figure}
$$\xymatrix{
&
\vdots
&
\\
\tau^{-1}P_1,\tau^{-1}P_2
\ar@{<->}[d]_{\mu_1}
\ar@{<->}[u]_{\mu_2}
\ar@{<->}[rr]_{\sigma}
&&
\tau^{-1}P_2,\tau^{-1}P_1 
\ar@{<->}[d]_{\mu_2}
\ar@{<->}[u]_{\mu_1}
\\
P_1,\tau^{-1}P_2
\ar@{<->}[d]_{\mu_2}
\ar@{<->}[rr]_{\sigma} &&
\tau^{-1}P_2,P_1
\ar@{<->}[d]_{\mu_1}
\\
P_1,P_2
\ar@{<->}[d]_{\mu_1}
\ar@{<->}[rr]_{\sigma}
&&
P_2,P_1
\ar@{<->}[d]_{\mu_2}
\\
P_1[1],P_2
\ar@{<->}[d]_{\mu_2}
\ar@{<->}[rr]_{\sigma}
&&
P_2,P_1[1]
\ar@{<->}[d]_{\mu_1}
\\
P_1[1],P_2[1]
\ar@{<->}[d]_{\mu_1}
\ar@{<->}[rr]_{\sigma}
&&
P_2[1],P_1[1]
\ar@{<->}[d]_{\mu_2}
\\
I_1,P_2[1]
\ar@{<->}[d]_{\mu_2}
\ar@{<->}[rr]_{\sigma}
&&
P_2[1],I_1
\ar@{<->}[d]_{\mu_1}
\\
I_1,I_2
\ar@{<->}[d]_{\mu_1}
\ar@{<->}[rr]_{\sigma}
&&
I_2,I_1
\ar@{<->}[d]_{\mu_2}
\\
\tau I_1,I_2
\ar@{<->}[d]_{\mu_2}
\ar@{<->}[rr]_{\sigma}
&&
I_2,\tau I_1
\ar@{<->}[d]_{\mu_1}
\\
\tau I_1,\tau I_2
\ar@{<->}[d]_{\mu_1}
\ar@{<->}[rr]_{\sigma}
&&
\tau I_2,\tau I_1
\ar@{<->}[d]_{\mu_2}
\\
&
\vdots
&
\\
}$$
\caption{The action of $M_2$ on $T_o^{\Lambda}$ for the Kronecker algebra $\Lambda$.}
\label{fig:M2action}
\end{figure}

The transitive braid group action on the set of complete exceptional sequences for a hereditary algebra arises in the following way~\cite[Lemma 9,Theorem]{cb-exc},~\cite[\S5,\S7]{ringel-exc}.
Write the braid group $B_n$ on $n$ strands in the usual way as
$$B_n=\left\langle \sigma_1,\ldots ,\sigma_{n-1}\,:\,\begin{array}{cc}
\sigma_i\sigma_j=\sigma_j\sigma_i, & |i-j|>1; \\
\sigma_i\sigma_j\sigma_i=\sigma_j\sigma_j\sigma_i, & |i-j|=1
\end{array}
\right\rangle.
$$
Given a complete exceptional sequence $(X_1,\ldots ,X_n)$ and $1\leq i\leq n$, there is a unique complete exceptional sequence of the form $(X_1,\ldots ,X_{i-1},X_{i+1},Y,X_{i+2},\ldots ,X_n)$ for some exceptional indecomposable module $Y$. The left version of the action of the braid group is given by: $$\sigma_i(X_1,\ldots ,X_n)=
(X_1,\ldots ,X_{i-1},X_{i+1},Y,X_{i+2},\ldots ,X_n).$$
If $n=2$, this means, in particular, that if $(X_1,X_2)$ is a complete exceptional sequence then there is an exceptional sequence of the form $(X_2,Y)$.

However, if $(X_1,X_2)$ is a complete $\tau$-exceptional sequence, it can be the case that there is no $\tau$-exceptional sequence of the form $(X_2,Y)$. This is not surprising, since the requirement on $X_2$ in the first case is that it is a $\tau$-rigid module, while the requirement in the second case is that it is a $\tau$-rigid object in the $\tau$-perpendicular category of $Y$. In the hereditary case, where a module $M$ is $\tau$-rigid if and only if $Ext^1(M,M) = 0$, a module in  
$J(Y) = Y^{\perpe}$ is $\tau$-rigid in $J(Y)$ if and only if it is $\tau$-rigid in $\module\Lambda$, but this is not true in general (see~\cite[\S1]{bm-exc}).  
We illustrate this point with the following example.

Let $\Gamma$ be the algebra given by the path algebra of the quiver
\raisebox{1pt}{$\xymatrix{1 \ar@<-0.5ex>[r]_{\alpha}  & 2 \ar@<-0.5ex>[l]_{\beta}}$}, subject to the relation $\beta\alpha=0$. The $\tau$-exceptional sequences over $\Gamma$ were given in~\cite[\S6.2]{bm-exc}. They are:
$$\left(1,\begin{smallmatrix} 2 \\ 1 \\ 2\end{smallmatrix}\right),
\left(2,\begin{smallmatrix} 1 \\ 2 \end{smallmatrix}\right),
\left(\begin{smallmatrix} 2 \\ 1 \end{smallmatrix},2\right),
\left(\begin{smallmatrix} 1 \\ 2 \end{smallmatrix},1\right).$$
We make the following observation.

\begin{lemma}
For the algebra $\Gamma$, which has complete $\tau$-exceptional sequences of length $2$, there is a $\tau$-rigid indecomposable module $X$ such that there is no complete $\tau$-exceptional sequence of the form $(X,Y)$ for some $\tau$-rigid module $Y$.
\end{lemma}
\begin{proof}
This can be seen from the list of $\tau$-exceptional sequences in $\module\Gamma$ above, since the module $P_2=\begin{smallmatrix} 2 \\ 1 \\ 2 \end{smallmatrix}$, despite being $\tau$-rigid, does not occur as the first term in any of the complete $\tau$-exceptional sequences in the list.
\end{proof}
Note that it follows that $P_2$ also does not occur as the first term in a complete signed $\tau$-exceptional sequence. 

The phenomenon described above also occurs with the right version~\cite{cb-exc,ringel-exc} of the action of the braid group. Given a complete exceptional sequence $(X_1,\ldots ,X_n)$ and $1\leq i\leq n$, there is a unique complete exceptional sequence of the form $(X_1,\ldots ,X_{i-1},Y,X_{i+1},X_{i+2},\ldots ,X_n)$ for some exceptional indecomposable module $Y$. The right version of the action of the braid group is given by: $$\sigma_i(X_1,\ldots ,X_n)=
(X_1,\ldots ,X_{i-1},Y,X_{i+1},X_{i+2},\ldots ,X_n).$$
If $n=2$, this means, in particular, that if $(X_1,X_2)$ is a complete exceptional sequence then there is an exceptional sequence of the form $(Y,X_1)$.

We can see from the list above that, although $(\begin{smallmatrix} 2 \\ 1 \end{smallmatrix},2)$ is a complete $\tau$-exceptional sequence for $\Gamma$, there is no such sequence of the form $(Y,\begin{smallmatrix} 2 \\ 1\end{smallmatrix})$. This is because, although $\begin{smallmatrix} 2 \\ 1 \end{smallmatrix}$ is $\tau$-rigid in $J(2)$, it is not a $\tau$-rigid $\Gamma$-module.

\section*{Acknowledgement}
The first named author would like to thank Eric Hanson for insightful conversations related to this work.

\end{document}